\documentclass[12pt]{amsart}
\usepackage[utf8]{inputenc}
\usepackage[english]{babel}
\usepackage{amsfonts}

\newtheorem{theorem}{Theorem}[section]
\newtheorem{lemma}[theorem]{Lemma}
\newtheorem{corollary}[theorem]{Corollary}

\theoremstyle{definition}   
\newtheorem{definition}{Definition}
\newtheorem{notation}{Notations}
\newtheorem{example}[theorem]{Example}

\theoremstyle{remark}
\newtheorem{remark}[theorem]{Remark}
\newtheorem{conjecture}[theorem]{Conjecture}

\numberwithin{equation}{section}

\usepackage{times}
\usepackage{enumerate}
\usepackage{tfrupee}
\usepackage{tikz}
\usetikzlibrary{chains,fit}
\usetikzlibrary{shapes,snakes}
\usetikzlibrary{graphs}
\usepackage{graphicx,adjustbox}
\usepackage{hyperref}
\usepackage{amsmath,amssymb}
\usepackage[all,cmtip]{xy}
 \usepackage{calrsfs} 
\usepackage{tabularx}

\usepackage{amscd}
\usepackage{graphicx}
\usepackage[all]{xy}

\title[Numerical Semigroups with unique Ap\'{e}ry expansions II]
{Numerical Semigroups with unique Ap\'{e}ry expansions II}
\author{
Joydip Saha
\and
Gaurab Tripathi
}
\date{}

\address{\small \rm  Stat-Math Unit, Indian Statistical Institute, 203 B.T. Road, Kolkata 700108.} 
\email{saha.joydip56@gmail.com}
\thanks{The first author acknowledges the receipt of post-doctoral fellowship from 
NBHM, Government of India.}

\address{\small \rm  Department of Mathematics, St. Xaviers' College(Autonomous), 30,Mother Teresa Sarani, Kolkata 700016, INDIA.}
\email{gaurabtripathi@sxccal.edu}

\date{}



\allowdisplaybreaks

\begin{document}
\begin{abstract}
In this paper, we carry out a fairly comprehensive 
study of  special classes of numerical semigroups, and their tangent cones, 
generated by the sequence of partial 
sums of an arithmetic progression , 
in embedding dimension $5$. These classes have unique expansions 
of the Ap\'{e}ry set elements. 
\end{abstract}

\maketitle

\section{Introduction}

The set $\Gamma$, subset of the set of nonnegative integers 
$\mathbb{N}$, is called a \textit{numerical semigroup} if it is 
closed under addition, contains zero and generates $\mathbb{Z}$ as 
a group. Every numerical semigroup $\Gamma$ satisfies the following two 
fundamental properties (see \cite{rs}): The complement $\mathbb{N}\setminus \Gamma$ is 
finite and $\Gamma$ has a unique minimal system of generators 
$a_{1} < \cdots < a_{n}$. The greatest integer not belonging to $\Gamma$, 
usually denoted by $F(\Gamma)$ is called the \textit{Frobenius number} 
of $\Gamma$. The integers $a_{1}$ and $n$, denoted by $m(\Gamma)$ and 
$e(\Gamma)$ respectively are known as the \textit{multiplicity} and the 
\textit{embedding dimension} of the semigroup $\Gamma$. The 
\textit{Ap\'{e}ry set} of $\Gamma$ with respect to a non-zero $a\in \Gamma$ is 
defined to be the set $\rm{Ap}(\Gamma,a)=\{s\in \Gamma\mid s-a\notin \Gamma\}$. 
 Given positive 
integers $a_{1} < \cdots < a_{n}$, every numerical semigroup ring 
$k[\Gamma] = k[t^{a_{1}}, \ldots , t^{a_{n}}]$ is the coordinate 
ring of an affine monomial curve given by the monomial parametrization 
$\nu : k[x_{1}, \ldots, x_{n}]\longrightarrow k[t]$, such that 
$\nu(x_{i}) = t^{a_{i}}$, $1\leq i\leq p$. The ideal $\ker(\nu)=\mathfrak{p}$ 
is the defining ideal of the parametrized monomial curve, which is graded 
with respect to the weighted gradation. 
\medskip

In this paper we continue the study of the numerical semigroup $\Gamma_{m}$ generated 
by the positive integers $s_{1},\ldots, s_{m}$, where $d>0$ and $a>0$ are integers with 
$\gcd(a,d)=1$, and $s_{n}=\frac{n}{2}[2a+(n-1)d]$, for $1\leq n\leq m$. For $m=4$ a details study has been done in \cite{pss1}.  It is known that uniqueness of 
representations of the Ap\'{e}ry set elements of a numerical semigroup is 
actually quite helpful to study tangent cone of the numerical semigroup; see \cite{rs1}, \cite{mss}. In general explicit writing of Ap\'{e}ry set of $\Gamma_{m}$ is very difficult. We show that for $m=5$ these classes has unique representation of elements of Ap\'{e}ry set w.r.t $a$ and study tangent cone of $\Gamma_{5}$. We also find  type, Frobenious number of $\Gamma_{5}$ and minimal generating set of the defining ideal. We conjectured that the elements Ap\'{e}ry set $\mathrm{AP}(\Gamma_{m},a)$ for $m\geq 6$ has always unique expansions see conjecture \ref{uniquecon} and we guess  $\mathrm{AP}(\Gamma_{6},a)$, see  conjecture \ref{apery6con}.
\medskip

\section{Ap\'{e}ry set of $\Gamma_{5}$}\label{cow}
We consider the numerical semigroup $ \Gamma_{m}$ generated 
by the positive integers $s_{1},\ldots, s_{m}$, where $d>0$ and $a>0$ are integers with 
$\gcd(a,d)=1$, and $s_{n}=\frac{n}{2}[2a+(n-1)d]$, for $1\leq n\leq m$. We denote the semigroup ring by $k[\Gamma_{m}]$ and the defining ideal 
by $\mathfrak{p}_{m}$.  
\medskip

\begin{theorem}\label{min4}
Let $a,d>0$ be integers with $\gcd(a,d)=1$. Let 
$s_{n}=\frac{n}{2}[2a+(n-1)d]$, for $1\leq n\leq 5$. The set 
$T_{5}=\{s_{1},\ldots, s_{5}\}$ is a minimal generating 
set for the numerical semigroup $\Gamma_{5} = \langle s_{1},\ldots, s_{5} \rangle$ if and only if $a\geq 11$.
\end{theorem}

\proof Suppose $a\geq 11$, it is known from \cite{pss1}, theorem $2.1$, that $T_{4}=\{s_{1},\ldots, s_{4}\}$ is a minimal generating set for the numerical semigroup $\Gamma_{4}$ if and only if $a\geq 7$.  Therefore we only need to check that $5a+10d\notin \Gamma_{4}$.
\smallskip

Suppose $$x_{1}a+x_{2}(2a+d)+x_{3}(3a+3d)+x_{4}(4a+6d)=5a+10d.$$ Simplifying we get, $a(5-x_{1}-2x_{2}-3x_{3}-4x_{4})=d(x_{2}+3x_{3}+6x_{4}-10)$. We have $gcd(a,d)=1$, therefore $ a|(x_{2}+3x_{3}+6x_{4}-10)$. Again $a\geq 11$ implies $(x_{2}+3x_{3}+6x_{4}-10)\geq 0 $, as R.H.S. of the above expression is positive, we get $ x_{1}+2x_{2}+3x_{3}+4x_{4}\leq 5$. Therefore $x_{2}\leq 2,x_{3},x_{4}\leq 1.$ This shows that $0\leq x_{2}+3x_{3}+6x_{4}-10\leq 2+3+6-10=1$. However this is true only if $x_{2}=2,x_{3}=1,x_{4}=1.$ But then  $x_{1}a+x_2(2a+d)+x_3(3a+3d)+x_4(4a+6d)=x_{1}a+11a+11d>5a+10d$ gives a contradiction. \qed
\medskip

We need a technical lemma, to finding Ap\'{e}ry set of $\Gamma_{5}$.
\medskip

\begin{lemma}\label{equals}
Let $a_{1},a_{2},a_{3},a_{4}$ be integers not all zero, such that $a_{1}\geq -2,a_{2}\geq -1,a_{4}\geq 0$ and $a_{1}+3a_{2}+6a_{3}=10a_{4}$. Then $4a_{1}+3a_{2}+5a_{4}\leq 0$ if and only if $(a_{1},a_{2},a_{3},a_{4})=(-2,0,2,1)$.
\end{lemma}
\proof We consider the equations,
\begin{eqnarray}
a_{1}+3a_{2}+6a_{3}=10a_{4}\label{1}
\end{eqnarray}
\begin{eqnarray}
4a_1+3a_2+5a_4\leq 0\label{2}
\end{eqnarray}

 From \ref{1} we get $a_{1}+a_{2}$ is even. So either $a_{2}\geq 0$ or $a_{1}\geq -1.$ Hence, $ 4a_{1}+3a_{2}\geq - 8$. Therefore $4a_{1}+3a_{2}+5a_{4}\leq 0$ implies $ a_{4}\leq 1$. We consider the following cases.
\begin{itemize}
\item \textbf{Case:1} Suppose $a_{4}=1.$ Now $a_{1}=-2$ implies $4a_{1}+3a_{2}+5a_{4}=3(a_{2}-1)>0$ for $a_2\neq 0 $ (as $2\mid a_1+a_2$). from the equation \ref{2} we get $a_{2}=0$. Therefore $ a_4=1,a_1=-2,a_2=0$ and this implies $ a_{3}=2$, which is in fact our solution. Again $a_{1} \geq -1$ implies $4a_{1}+5a_{4}\geq 1.$ Hence from the equation \ref{2} we get, $a_{2}<0$, therefore $a_{2}=-1$ and $ 6a_{3}=14$ from the equation \ref{1}. Hence we arrive at a contradiction.
\medskip

\item \textbf{Case:2} If $a_{4}=0$, then $a_1+3a_2+6a_3=0$. Therefore $3|a_{1}$ and we have $a_1\geq 0.$  From the equation \ref{2} we get $4a_{1}+3a_{2}<0.$ Now $a_{1}\geq 0$ implies $ a_{2}<0$, hence  $a_{2}=-1.$ But $2|a_1+a_2$ therefore $ a_{1}\geq 1$, hence $4a_{1}+3a_{2}>1>0$, again a contradiction.  
\end{itemize}
\qed
\medskip
\begin{notation}
Next we consider the following notations, for $n\in \mathbb{N}$, 
\begin{itemize}
\item $n=10q_{3n}+r_{3n}$, $0\leq r_{3n}\leq 9$,
\item  $r_{3n}=6q_{2n}+r_{2n}$, $0\leq r_{2n}\leq 5$,
\item  $r_{2n}=3q_{1n}+r_{1n}$, $0\leq r_{1n}\leq 2$. 
\item 
\begin{align*}
\mu(n)&=2r_{1n}+3q_{1n}+4q_{2n}+5q_{3n},\quad\mathrm{if}\,\, r_{1n}\neq 2 \,\,\mathrm{or}\,\, (r_{1n},q_{3n})=(2,0) \\
      &=2r_{1n}+3q_{1n}+4q_{2n}+5q_{3n}-1,\quad\mathrm{if}\,\, r_{1n}= 2 \,\,\mathrm{and}\,\, q_{3}>0\\ 
\end{align*}
\item $\phi(n)=\mu(n)a+nd$ 
\item $\omega(n)=\phi(n)-a$
\end{itemize}
 \end{notation}
 \begin{remark}\label{bound}
 We note that $0\leq q_{in}\leq 1$ for $i=1,2$ and $(q_{2n},q_{3n},q_{1n})\neq (1,1,1)$ as $0\leq r_{3n}\leq 9$. 
 \end{remark}
\begin{lemma}\label{guide}
  The set $S=\{\phi(n)\mid n=1,2,\cdots,a-1\,\,\rm{and}\,\, r_{1n}\neq 0\}\cup \{0\}$ is a subset of the semigroup $\Gamma_{5}$.
 \end{lemma}

\proof
We have, $$n=10q_{3n}+r_{3n}=10q_{3n}+6q_{2n}+r_{2n}=10q_{3n}+6q_{2n}+3q_{1n}+r_{1n}.$$ Therefore $$\mu(n)a+nd= r_{1n}(2a+d)+q_{1n}(3a+3d)+q_{2n}(4a+6d)+q_{3n}(5a+10d)\in \Gamma_{5}.$$  complete residue modulo $a$.  \qed

\begin{lemma}\label{penul}
 The set $\{\phi(n)|n=1,2,\cdots,a-1\,\rm{and}\, r_{1n}\neq 0\}\cup{0}$ is a subset of the Ap\'{e}ry set. 
\end{lemma}
\proof
Let $x_{0}(2a+d)+x_{1}(2a+3d)+x_2(4a+6d)+x_3(5a+10d)=m'a+nd$ be an element of the Ap\'{e}ry Set. Clearly, we take the coefficent of $a$ as $0$ since we are calculating w.r.t $a.$ Now we take $n$ to be such that $1\leq n\leq a-1$ and thus making it unique. Hence we get, $x_{0}+3x_{1}+6x_{2}+10x_{3}=n+ka$ for some $k\geq 0.$ Hence we get  $(x_{0}-r_1)+3(x_{1}-q_1)+6(x_{2}-q_2)+10(x_{3}-q_3)=ka.$ Replacing $x_0-r_1$ by $a_1$ and $x_i -q_i$ by $a_{i+1}$ for $i=1(1)3,$ we get, $a_1+3a_2+6a_3+10a_4=ka.$ Also $m'-m= 2a_1+3a_2+4a_3+5a_4+kd.$ Now we note, $r_{1}\leq 2, q_1\leq 1,q_2\leq 1$ and hence $a_1\geq -2, a_2,a_3\geq -1.$ We claim $k\geq 1\Rightarrow m'<m.$ To show this we will show $4a_1+6a_2+8a_3+10a_4>a_1+3a_2+6a_3+10a_4-ka.$ The latter is equivalent to showing, $3a_1+3a_2+2a_3>-ka.$ Now $a_1\geq 2, a_2,a_3\geq -1\Rightarrow 3a_1+3a_2+2a_3\geq -11.$ Now $-ka\leq -12k$ since we have taken $a>12.$\\
\noindent For $k=0,$ we note, $q_3$ being the quotient obtained in dividing $n$ by 10, we get, $x_3\leq q_3$ and hence $a_4\leq 0.$ Replacing $a_4$ by $-a_4$ we get $m=m'$ from Lemma \ref{equals}. The rest follows from Lemma \ref{guide} \qed

 \begin{theorem}
The Ap\'{e}ry set $\mathrm{Ap}( \Gamma_{5},a)$  of the semigroup $\Gamma_{5}$ w.r.t. element $a$ is given by $$\mathrm{Ap}( \Gamma_{5},a)=\{\phi(n) \ |\ n=1,2,\cdots,a-1\}\cup\{0\}$$ 
\end{theorem}
\proof
We note for $r_1=2,$ $\mu(n)a+nd=q_1(3a+3d)+(q_2+2)(4a+6d)+(q_3-1)(5a+10d).$ The rest follows essentially from Lemma \ref{equals} and Lemma \ref{penul}.\qed

\begin{lemma}\label{uniqueexp}
The expressions of each element in the Ap\'{e}ry Set is unique.
\end{lemma}
\proof
Suppose $x=s_0(2a+d)+s_1(3a+3d)+s_2(4a+6d)+s_3(5a+10d)$ be an element of the Ap\'{e}ry set with not all $a_{i}'s$ are zero, where $a_{i+1}=s_i-q_i$ for $i\geq 1$ and $a_{1}=s_0-r_1.$ Hence $a_1+3a_2+6a_3+10a_4=0.$ Then the result follows from Lemma \ref{equals}. \qed

\section{Fr\"{o}benious Number and Pseudo Fr\"{o}benious Numbers of $\Gamma_{5}$}

\begin{definition}
Let $\Gamma$ be a numerical semigroup, we say thet $x\in\mathbb{Z}$ is a 
\textit{pseudo-Frobenius number} if $x\notin \Gamma$ and $x+s\in \Gamma$ 
for all $s\in \Gamma\setminus \{0\}$. We denote by $\mathbf{PF}(\Gamma)$ 
and $\mathbf{F}(\Gamma)$ the set of  pseudo-Frobenius and Frobenius numbers of $\Gamma$ respectively. The cardinality of 
$\mathbf{PF}(\Gamma)$ is denoted by $t(\Gamma)$ and we call it the 
\textit{type} of $\Gamma$.
\end{definition}

Let $a,b\in \mathbb{Z}$. We define $\leq_{\Gamma}$ as $a\leq_{\Gamma} b$ 
if $b-a\in \Gamma$. This order relation defines a poset structure on 
$\mathbb{Z}$.

\begin{theorem}\label{max}
Let $\Gamma$ be a numerical semigroup and $a\in\Gamma\setminus \{0\}$. Then 
$\mathbf{PF}(\Gamma)=\{w-a\mid w\in \,\mathrm{Maximals}_{\leq_{\Gamma}}Ap(\Gamma, a)\}$. 
\end{theorem}
\proof See Proposition 8 in \cite{as}.\qed

\begin{lemma}
$\phi(5),\phi(8)$ are maximal elements of the Ap\'{e}ry set $\mathrm{Ap}(\Gamma_{5},a)$ for any $a\geq 12.$
\end{lemma}
\proof Let $\mu(n)a+nd\in \mathrm{Ap}(\Gamma_{5},a)$  such that $\phi(5)=7a+5d\leq_{\Gamma_{5}}\mu(n)a+nd$. Hence $(\mu(n)-7)a+(n-5)d\in \Gamma_{5}.$ We have $n=10q_{3n}+6q_{2n}+3q_{1n}+r_{1n}$ and  $0\leq q_{2n},q_{1n}\leq 1,0\leq r_{1n}\leq 2.$ Suppose $(q_{2n},q_{1n},r_{1n})=(0,0,0)$ then $n=10q_{3n}$ and $\mu(n)=5q_{3n}$. Now $n-5=10(q_{3n}-1)+0+3.1+2$, therefore $\mu(n-5)=5(q_{3n}-1)+3.1+3=5q_{3n}+1\neq \mu(n)-7$. Hence $(\mu(n)-7)a+(n-5)d\notin \Gamma_{5 }$ and gives a contradiction. Similarly we can show the other cases. To show $\phi(8)$ is a maximal element of $\mathrm{Ap}(\Gamma_{5},a)$, we proceed by same argument.\qed
\begin{lemma}
Let $a\geq 20.$ Let $a=10q+r$, $0\leq r\leq 9$. Let $\mathbf{PF}(\Gamma_{5})_{r}$ be the pseudo-Frobenius numbers of the numerical semigroup $\Gamma_{5}$ then we have  $\mathbf{PF}(\Gamma_{5})_{r}=\mathcal{T}_{r}\cup \{w(5),w(8)\}$. Where 
  
\begin{itemize}
\item  $\mathcal{T}_{0}=\{w(a-i)\mid i\in \{1,2,3,5,6\}\} $

\item $\mathcal{T}_{1}=\{w(a-i)\mid i\in \{1,2,3,4,6,7\}\}$

\item $\mathcal{T}_{2}=\{w(a-i)\mid i\in \{1,3,4,5,7,8\}\}$

\item $\mathcal{T}_{3}=\{w(a-i)\mid i\in \{1,2,4,5,6,7,8,9\}\} $
\item $\mathcal{T}_{4}=\{w(a-i)\mid i\in \{1,2,3,5,6,7,9,10\}\} $
\item $\mathcal{T}_{5}=\{w(a-i)\mid i\in \{1,3,6,7,8,10\}\} $
\item $\mathcal{T}_{6}=\{w(a-i)\mid i\in \{1,2,8,9\}\}  $
\item $\mathcal{T}_{7}=\{w(a-i)\mid i\in \{1,2,3,9,10\}\}$
\item $\mathcal{T}_{8}=\{w(a-i)\mid i\in \{1,3,4,10\}\}  $
\item $\mathcal{T}_{9}= \{w(a-i)\mid i\in \{1,2,4,5\}\} $
\end{itemize}
The Frobenius number is given by $w(a-1)$ for all cases.
\end{lemma}
\begin{proof}
We note that $\phi(n+10)- \phi(n)=5a+10d\in \Gamma_{5}.$ Hence $\phi(n+10)>_{\Gamma_{5}} \phi(n)$ for $n\geq 10$ and for $0\leq n<10$, we have $\phi(2)<_{\Gamma_{5}} \phi(4)<_{\Gamma_{5}} \phi(5).$ Also $\phi(n+10)>_{\Gamma_{5}} \phi(n)$ for $n\neq 2,5,8.$ Therefore the maximal elements are among 
$\{\phi(a-i)\mid 1\leq i\leq 10\}\cup\{\phi(5),\phi(8)\}.$ As indicated in the above lemma $\{\phi(5),\phi(8)\}$ are maximal elements of the Ap\'{e}ry Set. Now the result follows easily by inspecting each case. \\
For the Frobenius number we first note that $w(a-1)\in \mathcal{T}_{r} $ for each $0\leq r\leq 9.$  We know that Frobenius number is the maximum (w.r.t. natural order as in $\mathbb{N}$) among the pseudo Frobenius numbers, our job is done if we can show $w(a-1)$ is the maximum element of $\mathcal{T}_{r}$ for each $r.$ For this case we note that $w(a)>w(b)\Leftrightarrow \mu(a)>\mu(b).$ Now  \begin{equation}
\left\{
	\begin{array}{lll} 
                                           \mu(a-1)&=&[5(k-1)+7]a+(a-1)d\\
                                           \mu(a-2)&=&[5(k-1)+7]a+(a-2)d\\
                                           \mu(a-3)&=&[5(k-1)+6]a+(a-3)d\\
                                            \mu(a-4)&=&[5(k-1)+4]a+(a-4)d\\
                                             \mu(a-5)&=&[5(k-1)+6]a+(a-5)d\\
                                            \mu(a-6)&=&[5(k-1)+5]a+(a-6)d\\
                                           \mu(a-7)&=&[5(k-1)+3]a+(a-7)d\\
                                            \mu(a-8)&=&[5(k-1)+3]a+(a-8)d\\
                                            \mu(a-9)&=&[5(k-1)+2]a+(a-9)d\\
	\end{array}
\right.
\end{equation}
 Hence it easily follows that $w(a-1)$ is the maximum and hence the Frobenius number.
\end{proof}
\begin{lemma}

Let $12\leq a \leq 19.$ Then the pseudo-Frobenius numbers are given by \begin{equation}
\mathbf{PF}(\Gamma_{5})=\left\{
	\begin{array}{ll}
                                           \{w(9),w(10),w(5),w(8)\} & \mbox{if } a=11\\
		   		\{w(9),w(11),w(5),w(8)\} & \mbox{if } a=12\\
		   		\{w(9),w(11),w(12),w(5),w(8)\} & \mbox{if } a=13\\
		   		\{w(9),w(11),w(12),w(13),w(5),w(8)\} & \mbox{if } a=14\\
		   		\{w(9),w(12),w(14),w(5),w(8)\} & \mbox{if } a=15\\
		   		\{w(14),w(15),w(5),w(8)\} & \mbox{if } a=16\\
		   		\{w(14),w(15),w(16),w(5),w(8)\} & \mbox{if } a=17\\
		   		\{w(14),w(15),w(17),w(5),w(8)\} & \mbox{if } a=18\\
		   		\{w(14),w(15),w(17),w(18),w(5),w(8)\} & \mbox{if } a=19\\
	\end{array}
\right.
\end{equation}
The Frobenius numbers are given as \begin{equation}
\mathbf{F}(\Gamma_{5})=\left\{
	\begin{array}{ll}
                                          \{w(8)\} & \mbox{if } a=11,a>d\\
                                          \{w(9)\} & \mbox{if } a=11,d/2<a<d\\
                                          \{w(10)\} & \mbox{if } a=11,a<d/2\\
		   		\{w(8)\} & \mbox{if } a=12,a>3d\\
		   		\{w(11)\} & \mbox{if } a=12,a<3d\\
		   		\{w(12)\} & \mbox{if } a=13\\
		   		\{w(13)\} & \mbox{if } a=14\\
		   		\{w(14)\} & \mbox{if } a=15\\
		   		\{w(15)\} & \mbox{if } a=16\\
		   		\{w(15)\} & \mbox{if } a=17,2a>d\\
		   		\{w(16)\} & \mbox{if } a=17,2a<d\\
		   		\{w(17)\} & \mbox{if } a=18\\
		   		\{w(18)\} & \mbox{if } a=19\\
	\end{array}
\right.
\end{equation}

\end{lemma}
	
\begin{proof}
Results follow easily by considering each case. 
\end{proof}
\section{Minimal generating set for the defining ideals}
Let us begin with the following theorem from \cite{g}, which helps us compute 
a minimal generating set for the defining ideal of a monomial curve.
\medskip

\begin{theorem}\label{gastinger}
 Let $A = k[x_{1},\ldots,x_{n}]$ be a polynomial ring, $I\subset A$ the defining
ideal of a monomial curve defined by natural numbers $a_{1},\ldots,a_{n}$, 
whose greatest common divisor is $1$.  Let $J \subset I$ be a subideal. 
Then $J = I$ if and only if $\mathrm{dim}_{k} A/\langle J + (x_{i}) \rangle =a_{i}$
for some $i$. (Note that the above conditions are also equivalent to 
$\mathrm{dim}_{k} A/\langle J + (x_{i}) \rangle =a_{i}$ for any $i$.)
\end{theorem}  

\proof See \cite{g}.\qed
\medskip

\begin{lemma}\label{equal}
Let $A = k[x_{1},\ldots,x_{n}]$ be a polynomial ring. For a monomial ideal $J$ of $A$, we write 
the unique minimal generating set of $J$ as $G(J)$. Let $I=\langle f_{1},\ldots f_{k}\rangle$ and 
$I_{i}=\langle f_{1},\ldots,\hat{f}_{i},\ldots f_{k}\rangle$, $1\leq i\leq k$. 
Suppose that, with respect to some monomial order on $A$, $\{ \mathrm{LT}(f_{1}),\ldots \mathrm{LT}(f_{k})\} \subset G(\mathrm{LT}(I)) $ and 
$G(\mathrm{LT}(I_{i})) \subset G(\mathrm{LT}(I))\setminus \{\mathrm{LT}(f_{i})\}$ for all $1\leq i\leq k$. 
Then $I$ is minimally generated by $\{f_{1},\ldots f_{k}\}$.
\end{lemma}

\proof see  \cite{pss1} \qed
\medskip
\begin{notation}\label{poly}
Let us define a set of polynomials $\mathcal{G}=\{g_{1},\ldots ,g_{7}\}$ by
\begin{itemize}
\item $g_{1}=x_{4}^4-x_{1}x_{2}x_{3}x_{5}^{2}$
\item $g_{2}= x_{3}x_{4}^{3}-x_{1}^{3}x_{2}x_{5}^{2}$
\item $g_{3}=x_{3}^{2}-x_{1}^{2}x_{4}$
\item $g_{4}=x_{2}x_{4}^{2}-x_{1}^{2}x_{3}x_{5}$
\item $g_{5}=x_{2}x_{3}x_{4}-x_{1}^{4}x_{5}$
\item $g_{6}=x_{2}^{3}-x_{1}^{3}x_{3}$
\item $g_{7}=x_{1}x_{4}^{2}-x_{2}^{2}x_{5}$
\end{itemize}
Let $a=10q+r$, $0\leq r\leq 9$ and $d\geq 1$ with $\gcd(a,d)=1$. If $a\geq 19$ and $(a,d)\notin\{(21,1),(21,2)\}$ , we define set of polynomials $\mathcal{H}_{r}$ for $0\leq r\leq 9$ by
\begin{enumerate}
\item  $\mathcal{H}_{0}=\{h_{01},h_{02}\}$, where
\begin{itemize}
\item $h_{01}=x_{1}^{5q+d}-x_{5}^{q}$
\item $h_{02}=x_{1}^{5q+d-1}x_{2}^{2}-x_{4}^{2}x_{5}^{q-1}$
\smallskip

\end{itemize}
\item $\mathcal{H}_{1}=\{h_{11},\ldots,h_{15}\} $, we have
\smallskip

\begin{itemize}
\item $h_{11}=x_{1}^{5q+d-13}x_{2}^{9}-x_{5}^{q+1} $,
\item $h_{12}= x_{1}^{5q+d-6}x_{2}^{5}-x_{4}x_{5}^{q}$,
\item $h_{13}=x_{1}^{5q+d-1}x_{2}^{2}-x_{3}x_{5}^{q}$,
\item $h_{14}=x_{1}^{5q+d+2}-x_{2}x_{5}^{q}$,
\item $h_{15}=x_{1}^{5q+d+1}x_{2}-x_{4}^{2}x_{5}^{q-1}$
\end{itemize}
\smallskip

\item $\mathcal{H}_{2}=\{h_{21},\ldots,h_{24}\} $, where
\smallskip

\begin{itemize}
\item $h_{21}=x_{1}^{5q+d-11}x_{2}^{8}-x_{5}^{q+1} $,
\item $h_{22}= x_{1}^{5q+d-4}x_{2}^{4}-x_{4}x_{5}^{q}$,
\item $h_{23}=x_{1}^{5q+d+1}x_{2}-x_{3}x_{5}^{q}$,
\item $h_{24}=x_{1}^{5q+d+3}-x_{4}^{2}x_{5}^{q-1}$,

\end{itemize}
\smallskip

\item $\mathcal{H}_{3}=\{h_{31},\ldots,h_{35}\} $, we have
\smallskip

\begin{itemize}
\item $h_{31}=x_{1}^{5q+d-9}x_{2}^{7}-x_{5}^{q+1} $,
\item $h_{32}= x_{1}^{5q+d-2}x_{2}^{3}-x_{4}x_{5}^{q}$,
\item $h_{33}=x_{1}^{5q+d+3}-x_{3}x_{5}^{q}$,
\item $h_{34}=x_{1}^{5q+d-3}x_{2}^{5}-x_{4}^{3}x_{5}^{q-1}$,
\item $h_{35}=x_{1}^{5q+d+2}x_{2}^{2}-x_{3}x_{4}^{2}x_{5}^{q-1}$
\end{itemize}
\smallskip

\item $\mathcal{H}_{4}=\{h_{41},\ldots,h_{45}\} $, we have
\smallskip

\begin{itemize}
\item $h_{41}=x_{1}^{5q+d-7}x_{2}^{6}-x_{5}^{q+1} $,
\item $h_{42}= x_{1}^{5q+d}x_{2}^{2}-x_{4}x_{5}^{q}$,
\item $h_{43}=x_{1}^{5q+d+5}-x_{2}x_{3}x_{5}^{q}$,
\item $h_{44}=x_{1}^{5q+d-1}x_{2}^{4}-x_{4}^{3}x_{5}^{q-1}$,
\item $h_{45}=x_{1}^{5q+d+4}x_{2}-x_{3}x_{4}^{2}x_{5}^{q-1}$
\end{itemize}
\smallskip

\item $\mathcal{H}_{5}=\{h_{51},\ldots,h_{54}\} $, we have
\smallskip

\begin{itemize}
\item $h_{51}=x_{1}^{5q+d-5}x_{2}^{5}-x_{5}^{q+1} $,
\item $h_{52}= x_{1}^{5q+d+2}x_{2}-x_{4}x_{5}^{q}$,
\item $h_{53}=x_{1}^{5q+d+1}x_{2}^{3}-x_{4}^{3}x_{5}^{q-1}$,
\item $h_{54}=x_{1}^{5q+d+6}-x_{3}x_{4}^{2}x_{5}^{q-1}$
\end{itemize}
\smallskip

\item $\mathcal{H}_{6}=\{h_{61},\ldots,h_{63}\} $, we have
\smallskip

\begin{itemize}
\item $h_{61}=x_{1}^{5q+d-3}x_{2}^{4}-x_{5}^{q+1} $,
\item $h_{62}= x_{1}^{5q+d+4}-x_{4}x_{5}^{q}$,
\item $h_{63}=x_{1}^{5q+d+3}x_{2}^{2}-x_{4}^{3}x_{5}^{q-1}$,

\end{itemize}
\smallskip

\item $\mathcal{H}_{7}=\{h_{71},\ldots,h_{75}\} $, we have
\smallskip

\begin{itemize}
\item $h_{71}=x_{1}^{5q+d-1}x_{2}^{3}-x_{5}^{q+1} $,
\item $h_{72}= x_{1}^{5q+d-2}x_{2}^{5}-x_{4}^{2}x_{5}^{q}$,
\item $h_{73}=x_{1}^{5q+d+3}x_{2}^{2}-x_{3}x_{4}x_{5}^{q}$,
\item $h_{74}=x_{1}^{5q+d+6}-x_{2}x_{4}x_{5}^{q}$,
\item $h_{75}=x_{1}^{5q+d+5}x_{2}-x_{4}^{2}x_{5}^{q-1}$
\end{itemize}
\smallskip

\item $\mathcal{H}_{8}=\{h_{81},\ldots,h_{84}\} $, we have
\smallskip

\begin{itemize}
\item $h_{81}=x_{1}^{5q+d+1}x_{2}^{2}-x_{5}^{q+1} $,
\item $h_{82}=x_{1}^{5q+d+5}x_{2}-x_{3}x_{4}x_{5}^{q}$,
\item $h_{83}=x_{1}^{5q+d}x_{2}^{4}-x_{4}^{2}x_{5}^{q}$,
\item $h_{84}=x_{1}^{5q+d+7}-x_{4}^{3}x_{5}^{q-1}$
\end{itemize}
\smallskip

\item $\mathcal{H}_{9}=\{h_{91},\ldots,h_{93}\} $, where
\smallskip

\begin{itemize}
\item $h_{91}=x_{1}^{5q+d+3}x_{2}-x_{5}^{q+1} $,
\item $h_{92}=x_{1}^{5q+d+7}-x_{3}x_{4}x_{5}^{q}$,
\item $h_{93}=x_{1}^{5q+d+2}x_{2}^{3}-x_{4}^{2}x_{5}^{q}$,

\end{itemize}
\smallskip
\end{enumerate}
\end{notation}
\begin{theorem}\label{gen}
Let $a=10q+r$, $0\leq r\leq 9$ and $d\geq 1$ with $\gcd(a,d)=1$. If $a\geq 19$ and $(a,d)\notin\{(21,1),(21,2)\}$ , then $\mathcal{G}\cup \mathcal{H}_{r}$ is the minimal generating set for the defining ideal $\mathfrak{p}_{5}$ of the semigroup $\Gamma_{5}$.
\end{theorem}
\proof we proceed by the theorem \ref{gastinger}. Let $A=k[x_{1},\ldots,x_{5}]$ and  $J=\mathfrak{p}_{5}+\langle x_{1}\rangle $ and Mon$(A)$ denotes the set of all monomials of $A$. We need to show that $\mathrm{dim}_{k}B=a$, where $B=k[x_{1},\ldots,x_{5}]/J $. Suppose $a=10q+1$, We note that $J$ is a monomial ideal and  
$$G(J)=\{x_{1}, x_{4}^{4},x_{3}x_{4}^{3},x_{3}^{2},x_{2}x_{4}^{2},x_{2}x_{3}x_{4},x_{2}^{3},x_{2}^{2}x_{5},x_{5}^{q+1},x_{4}x_{5}^{q},x_{3}x_{5}^{q},x_{2}x_{5}^{q},x_{4}^{2}x_{5}^{q-1} \}$$  where $G(J)$ is the unique minimal generating set of $J$. We consider the canonical mapping $\pi:A\rightarrow B$ and the set $S=\{m\in \mathrm{Mon}(A) \mid p\nmid m, p\in G(J)\}$. Then $\pi(S)$ forms a basis set of $B$ over $k$. Let us list those monomials in disjoint sets,
\begin{itemize}
\item $\mathcal{B}_1= \{1\}$
\item $\mathcal{B}_2= \{x_{2},x_{2}^{2},x_{3},x_{4},x_{4}^{2},x_{4}^{3},x_{5},x_{5}^{2},\ldots,x_{5}^{q}\}$
\item $\mathcal{B}_3= \{x_{2}x_{3},x_{2}^{2}x_{3},x_{2}x_{4},x_{2}^{2}x_{4},x_{2}x_{5},x_{2}x_{5}^{2},\ldots,x_{2}x_{5}^{q-1}\}$
\item $\mathcal{B}_4= \{x_{3}x_{4},x_{3}x_{4}^{2}\}$
\item $\mathcal{B}_5= \{x_{3}x_{5},x_{3}x_{5}^{2},\ldots,x_{3}x_{5}^{q-1}\}$
\item $\mathcal{B}_6= \{x_{4}x_{5},\ldots,x_{4}x_{5}^{q-1}\}$
\item $\mathcal{B}_7= \{x_{4}^{2}x_{5},\ldots,x_{4}^{2}x_{5}^{q-2}\}$
\item $\mathcal{B}_8= \{x_{4}^{3}x_{5},x_{4}^{3}x_{5}^{2},\ldots,x_{4}^{3}x_{5}^{q-2}\}$
\item $\mathcal{B}_9= \{x_{2}x_{3}x_{5},\ldots,x_{2}x_{3}x_{5}^{q-1}\}$
\item $\mathcal{B}_{10}= \{x_{3}x_{4}x_{5},\ldots,x_{3}x_{4}x_{5}^{q-1}\}$
\item $\mathcal{B}_{11}= \{x_{2}x_{4}x_{5},\ldots,x_{2}x_{4}x_{5}^{q-1}\}$
\item $\mathcal{B}_{12}= \{x_{3}x_{4}^{2}x_{5},\ldots,x_{3}x_{4}^{2}x_{5}^{q-2}\}$

\end{itemize}
We have $S=\displaystyle\sum_{i=1}^{12}\mathcal{B}_{i}$, therefore  
\begin{align*}
\mid S\mid =\displaystyle\sum_{i=1}^{12}\mid B_{i}\mid &= 1+(6+q)+(4+q-1)+(2+q-1)+(q-1)+(q-2)+(q-1)\\
&+(q-1)+(q-1)+(q-2)+(q-2)+(q-1)+(q-2)\\
&=10q+1 =a.
\end{align*}
Hence by \ref{gastinger} the set $\mathcal{G}\cup \mathcal{H}_{1}$ is the generating set for the defining ideal $\mathfrak{p}_{5}$ of the semigroup $\Gamma_{5}$. Minimality of the generating set follows from the lemma \ref{equal}. The Other cases can be proved by similar argument.\qed
\bigskip

\begin{notation}
 Let us define some set of polynomials $\mathcal{T}_{(a,d)} $ in $k[x_{1},\ldots,x_{5}]$ for particular values of $a$ and $d$,
\begin{enumerate}
\item $\mathcal{T}_{(11,d)}=\{p_{11},p_{12},p_{13},p_{14}\} $ where,
\begin{itemize}
\item $p_{11}=x_{1}^{d-1}x_{2}^{5}-x_{4}x_{5}$
\item $p_{12}=x_{1}^{d+4}x_{2}^{2}-x_{3}x_{5}$
\item $p_{13}=x_{1}^{d+7}-x_{2}x_{5}$
\item $p_{14}=x_{1}^{d+6}x_{2}-x_{4}^{2}$

\end{itemize}
\item  $\mathcal{T}_{(12,d)}=\{p_{21},p_{22},p_{23}\} $ where,
\begin{itemize}
\item $p_{21}=x_{1}^{d+1}x_{2}^{4}-x_{4}x_{5}$
\item $p_{22}=x_{1}^{d+6}x_{2}-x_{3}x_{5}$
\item $p_{23}=x_{1}^{d+8}-x_{4}^{2}$
\end{itemize}
\item  $\mathcal{T}_{(13,d)}=\{p_{31},p_{32},p_{33},p_{34}\} $ and
\begin{itemize}
\item $p_{31}=x_{1}^{d+3}x_{2}^{3}-x_{4}x_{5}$
\item $p_{32}=x_{1}^{d+8}-x_{3}x_{5}$
\item $p_{33}=x_{1}^{d+2}x_{2}^{5}-x_{4}^{3}$
\item $p_{34}=x_{1}^{d+7}x_{2}^{2}-x_{3}x_{4}^{2}$
\end{itemize}
\item  $\mathcal{T}_{(14,d)}=\{p_{41},p_{42},p_{43},p_{44}\} $ and
\begin{itemize}
\item $p_{41}=x_{1}^{d+5}x_{2}^{2}-x_{4}x_{5}$
\item $p_{42}=x_{1}^{d+10}-x_{2}x_{3}x_{5}$
\item $p_{43}=x_{1}^{d+4}x_{2}^{4}-x_{4}^{3}$
\item $p_{44}=x_{1}^{d+9}x_{2}-x_{3}x_{4}^{2}$
\end{itemize}
\item  $\mathcal{T}_{(15,d)}=\{p_{51},p_{52},p_{53},p_{54}\} $ where,
\begin{itemize}
\item $p_{51}=x_{1}^{d}x_{2}^{5}-x_{5}^{2}$
\item $p_{52}=x_{1}^{d+7}x_{2}-x_{4}x_{5}$
\item $p_{53}=x_{1}^{d+6}x_{2}^{3}-x_{4}^{3}$
\item $p_{54}=x_{1}^{d+11}-x_{3}x_{4}^{2}$
\end{itemize}
\item  $\mathcal{T}_{(16,d)}=\{p_{61},p_{62},p_{63}\} $ where,
\begin{itemize}
\item $p_{61}=x_{1}^{d+2}x_{2}^{4}-x_{5}^{2}$
\item $p_{62}=x_{1}^{d+9}-x_{4}x_{5}$
\item $p_{63}=x_{1}^{d+8}x_{2}^{2}-x_{4}^{3}$
\end{itemize}
\item  $\mathcal{T}_{(17,d)}=\{p_{71},p_{72},p_{73},p_{74},p_{75}\} $ where,
\begin{itemize}
\item $p_{71}=x_{1}^{d+4}x_{2}^{3}-x_{5}^{2}$
\item $p_{72}=x_{1}^{d+3}x_{2}^{5}-x_{4}^{2}x_{5}$
\item $p_{73}=x_{1}^{d+8}x_{2}^{2}-x_{3}x_{4}x_{5}$
\item $p_{74}=x_{1}^{d+11}-x_{2}x_{4}x_{5}$
\item $p_{75}=x_{1}^{d+10}x_{2}-x_{4}^{3}$
\end{itemize}
\item  $\mathcal{T}_{(18,d)}=\{p_{81},p_{82},p_{83},p_{84}\} $ where,
\begin{itemize}
\item $p_{81}=x_{1}^{d+6}x_{2}^{2}-x_{5}^{2}$
\item $p_{82}=x_{1}^{d+10}x_{2}-x_{3}x_{4}x_{5}$
\item $p_{83}=x_{1}^{d+5}x_{2}^{4}-x_{4}^{2}x_{5}$
\item $p_{84}=x_{1}^{d+12}-x_{4}^{3}$
\end{itemize}
\medskip

\item Let us consider the following sets.
\begin{itemize}
\item $\mathfrak{T}_{(11,1)}=\mathcal{T}_{(11,1)}\cup\{g_{3},g_{5},g_{6}, x_{2}^{5}-x_{4}x_{5}\}$
\item $\mathfrak{T}_{(11,2)}=\mathcal{T}_{(11,2)}\cup\{g_{3},g_{5},g_{6}, x_{2}^{3}x_{3}^{2}-x_{5}^{2}\}$
\item $\mathfrak{T}_{(11,3)}=\mathcal{T}_{(11,3)}\cup\{g_{3},g_{5},g_{6}, x_{1}x_{2}^{3}x_{3}^{2}-x_{5}^{2}\}$
\item $\mathfrak{T}_{(11,4)}=\mathcal{T}_{(11,4)}\cup\{g_{3},g_{5},g_{6}, x_{1}^{2}x_{2}^{3}x_{3}^{2}-x_{5}^{2}\}$
\item $\mathfrak{T}_{(11,5)}=\mathcal{T}_{(11,5)}\cup\{g_{3},g_{5},g_{6}, x_{2}^{6}x_{3}-x_{5}^{2}\}$
\item $\mathfrak{T}_{(11,d)}=\mathcal{T}_{(11,d)}\cup\{g_{3},g_{5},g_{6}, x_{1}x_{2}^{6}x_{3}-x_{5}^{2}\}$, where $d=6,7$
\item $\mathfrak{T}_{(11,d)}=\mathcal{T}_{(11,d)}\cup\{g_{3},g_{5},g_{6}, x_{1}^{d-8}x_{2}^{9}-x_{5}^{2}\}$, where $d\geq 8$.
\end{itemize}
\item Let 
\begin{itemize}
\item $\mathfrak{T}_{(12,1)}=\mathcal{T}_{(12,1)}\cup\{g_{3},g_{5},g_{6},g_{7}, x_{1}x_{2}^{2}x_{3}^{2}-x_{5}^{2}\}$
\item $\mathfrak{T}_{(12,5)}=\mathcal{T}_{(12,5)}\cup\{g_{3},g_{5},g_{6},g_{7}, x_{1}^{2}x_{2}^{5}x_{3}-x_{5}^{2}\}$
\item $\mathfrak{T}_{(12,d)}=\mathcal{T}_{(12,d)}\cup\{g_{3},g_{5},g_{6},g_{7}, x_{1}^{d-6}x_{2}^{8}-x_{5}^{2}\}$, for $d\geq 6$.

\end{itemize}
\medskip

\item Suppose
\begin{itemize}
\item $\mathfrak{T}_{(13,1)}=\mathcal{T}_{(13,1)}\cup\{g_{3},g_{4},g_{5},g_{6},g_{7}, x_{2}^{2}x_{3}-x_{5}^{2}\}$
\item $\mathfrak{T}_{(13,2)}=\mathcal{T}_{(13,2)}\cup\{g_{3},g_{4},g_{5},g_{6},g_{7}, x_{1}x_{2}^{4}x_{3}-x_{5}^{2}\}$
\item $\mathfrak{T}_{(13,3)}=\mathcal{T}_{(13,3)}\cup\{g_{3},g_{4},g_{5},g_{6},g_{7}, x_{1}^{2}x_{2}^{4}x_{3}-x_{5}^{2}\}$
\item $\mathfrak{T}_{(13,d)}=\mathcal{T}_{(13,d)}\cup\{g_{3},g_{4},g_{5},g_{6},g_{7}, x_{1}^{d-4}x_{2}^{7}-x_{5}^{2}\}$, for $d\geq 4$.
\end{itemize}
\medskip

\item Suppose
\begin{itemize}
\item $\mathfrak{T}_{(14,1)}=\mathcal{T}_{(14,1)}\cup\{g_{3},g_{4},g_{5},g_{6},g_{7}, x_{2}^{4}x_{3}-x_{5}^{2}\}$
\item $\mathfrak{T}_{(14,d)}=\mathcal{T}_{(14,d)}\cup\{g_{3},g_{4},g_{5},g_{6},g_{7}, x_{1}^{d-2}x_{2}^{6}-x_{5}^{2}\}$, for $d\geq 2$.
\end{itemize}
\item Let 
\begin{itemize}
\item $\mathfrak{T}_{(a,d)}=\mathcal{T}_{(a,d)}\cup\{g_{3},g_{4},g_{5},g_{6},g_{7}\}$, for $15\leq a\leq 18$ and $d\geq 1$
\end{itemize}
\item Let,
\begin{itemize}
\item $\mathfrak{T}_{(21,d)}=(\mathcal{H}_{1}\setminus \{h_{11}\})\cup \{x_{1}x_{2}^{6}x_{3}-x_{5}^{3}\}$, for $d=1,2$
\end{itemize}
\end{enumerate}
\end{notation}
\begin{theorem}
Let $11\leq a\leq 18$ and $d\geq 1$ with $\gcd(a,d)=1$ or  $(a,d)\in\{(21,1),(21,2)\}$ , then $\mathfrak{T}_{(a,d)}$ is the minimal generating set for the defining ideal $\mathfrak{p}_{5}$ of the semigroup $\Gamma_{5}$.
\end{theorem}
\proof Similar proof as theorem \ref{gen}.\qed
\medskip

\begin{remark}
 Minimal generating set of the defining ideal $\mathfrak{p}_{5}$ can be found abstractly in \cite{rs1}, when elements of  Ap\'{e}ry set has unique representation. But here we have written explicitly.
\end{remark}
\medskip
\section{Ap\'{e}ry table and Tangent Cone of $k[[\Gamma_{5}]]$}
Throughout this section we assume that the field $k$ is infinite.
\medskip

\begin{definition} Let $(R,\mathfrak{m})$ be a Noetherian local ring and $I\subset R$ 
be an ideal of $R$. The fibre cone of $I$ is the ring 
$$F(I)=\displaystyle\bigoplus_{n\geq 0}\dfrac{I^{n}}{\mathfrak{m}I^{n}}\cong R[It]\otimes R/\mathfrak{m}.$$ 
Krull dimension of the ring $F(I)$ is called the \textit{analytic spread} of 
$I$, denoted by $\ell(I)$.
\end{definition}
\medskip

An ideal $J\subset I$ is called a \textit{reduction} of $I$ if there exists an 
integer $n>0$ such that $I^{n+1}=JI^{n}$. A reduction $J$ of $I$ is a \textit{minimal reduction} 
if $J$ is minimal with respect to inclusion among reductions of $I$. 
A minimal reduction always exists by \cite{nr}. It is well known that $J$ is a 
minimal reduction of $I$ if and only if $J$ is minimally generated by $\ell(I)$ 
number of elements, i.e, $\mu(J)=\ell(I)$. If  $J$ is a minimal reduction of $I$, 
then the least integer $r$ such that $I^{r+1}=JI^{r}$, is the reduction number 
of $I$ with respect to $J$, denoted by $r_{J}(I)$.
\medskip

We are interested in the semigroup ring $k[[\Gamma_{5}]]$, 
which is the coordinate ring of the algebroid monomial curve defined by 
the numerical semigroup $\Gamma_{5}$. Let  $a\geq 11$ and $d > 0$ be two integers, 
such that $\gcd(a,d)=1$. Let $R=k[[t^{a},t^{2a+d},t^{3a+3d},t^{4a+6d}, t^{5a+10d}]]$ and 
$\mathfrak{m}$ is the maximal ideal $\langle t^{a},t^{2a+d},t^{3a+3d},t^{4a+6d}, t^{5a+10d}\rangle$. 
Consider the principal ideal $I=\langle t^{a}\rangle \subset R$. The fibre cone of 
$I$ is the ring 
$$F(I)=\displaystyle\bigoplus_{n\geq 0}\dfrac{I^{n}}{\mathfrak{m}I^{n}}\cong R 
[It]\otimes R/\mathfrak{m}.$$ 
We note that here $\ell(I)=1$ and the tangent cone $G_{\mathfrak{m}}=\displaystyle\bigoplus_{n\geq 0}\dfrac{\mathfrak{m}^{n}}{\mathfrak{m}^{n+1}}$ is an $F(I)$-algebra. Moreover $F(I)\hookrightarrow G_{\mathfrak{m}} $ is a Noether normalisation (see \cite{cz1} and \cite{cz2}).
\medskip

Suppose $\Gamma$ be a numerical semigroup minimally generated by $a_{1}<\cdots 
<a_{e}$. Let $M=\Gamma\setminus\{0\}$ and for a positive integer $n$, we write 
$nM:=M+\cdots+M$ ($n$-copies). Let $\mathfrak{m}$ be the maximal ideal of the 
ring $k[[t^{a_{1}},\ldots t^{a_{e}}]]$. Then $(n+1)M=a+nM$ for all $n\geq r$ 
if and only if $r=r_{(t^{a_{1}})}(\mathfrak{m})$.
\medskip

Let $\mathrm{Ap}(\Gamma,a_{1})=\{0,\omega_{1},\ldots,\omega_{a_{1}-1}\}$. Now for 
each $n\geq 1$, let us define $\mathrm{Ap}(nM)=\{\omega_{n,0},\ldots\omega_{n,a_{1}-1}\}$ 
inductively. We define $\omega_{1,0}=a_{1}$ and $\omega_{1,i}=\omega_{i}$, for $1\leq i\leq a_{1}-1$. 
Then $\mathrm{Ap}(M)=\{a_{1},\omega_{1},\ldots,\omega_{a_{1}-1}\}$. Now we define 
$\omega_{n+1,i}=\omega_{n,i}$, if $\omega_{n,i}\in (n+1)M $, and $\omega_{n+1,i}=\omega_{n,i}+a_{1}$, 
otherwise. We note that $\omega_{n+1,i}=\omega_{n,i}+a_{1} $ for all $0\leq i\leq a_{1}-1$ and $n\leq r_{(t^{a_{1}})}(\mathfrak{m})$. Then, the Ap\'{e}ry table $\mathrm{AT}(\Gamma,a_{1})$ of $\Gamma$ is a table 
of size $(r_{(t^{a_{1}})}(\mathfrak{m})+1)\times a_{1}$, whose $(0,t)$ entry is $\omega_{t}$, 
for $0\leq t\leq {a_{1}-1}$ (we take $\omega_{0}=0$), and the $(s,t)$ entry is $\omega_{st}$, 
for $1\leq s\leq r_{(t^{a_{1}})}(\mathfrak{m})$ and $ 0\leq t\leq {a_{1}-1}$.
\medskip

\medskip
  
\begin{notation} We consider the following notations.
\begin{itemize}
\item Let $t_k$ denotes the number of elements in the Ap\'{e}ry set  $\mathrm{Ap}(\Gamma_{5},a)$ whose order is  $k$. We note that $t_{1}=4$ as $\Gamma_{5}$ is minimally generated by $5$ elements.
\item  Let $\psi:\mathbb{N}\rightarrow \mathbb{N}$ be a function defined as $\psi(2)=-1, \psi(3)=2,\psi(i)=0$ for $i\neq 2,3. $
\end{itemize}

\end{notation}
\begin{lemma}\label{number}
Let $a=10q+r$, $0\leq r\leq 9$. For $k\geq 2$ we have
\smallskip

\begin{center}
\begin{tabular}{ |c|c|c| } 
 \hline
 $k$ & $r$ & $t_k $\\ 
\hline
$<q$ &$ \geq 0$ & $10+\psi(k)$ \\ 
\hline
 $q$ & $>0$ & $10+\psi(k)$ \\ 
 $q$ & $ 0$ & $9+\psi(k)$ \\ 
\hline
 $q+1$ & $0$ & $5+\psi(k)$ \\ 
  & $1$ & $5+\psi(k)$ \\ 
  & $2$ & $6+\psi(k)$ \\ 
   & $3$ & $7+\psi(k)$ \\ 
  & $4,5,6$ & $8+\psi(k)$ \\ 
  & $7,8,9$ & $9+\psi(k)$ \\ 
 \hline
 $q+2$ & $0,1,2,3,4$ & $\psi(k)$ \\ 
 & $5$ & $1+\psi(k)$ \\ 
 & $6$ & $2+\psi(k)$ \\ 
 & $7$ & $2+\psi(k)$ \\ 
 & $8$ & $3+\psi(k)$ \\ 
 & $9$ & $4+\psi(k)$ \\ 
\hline
\end{tabular}
\end{center}

\end{lemma}
\begin{proof}
Let for some $1\leq n\leq a-1$, $o(\mu(n)a+nd)=k.$  We refer the generators of $\Gamma_5$ as $s_{1}<s_{2}<s_{3}<s_{4}<s_{5}$ as in section \ref{cow}. We note that, 
\begin{align*}
\mu(n)a+nd &=q_{1n}s_{3}+(q_{2n}+2)s_{4}+(q_{3n}-1)s_{5}&\mathrm{if}\quad r_{1n}= 2 \,\,\mathrm{and}\,\, q_{3}>0\\
&=r_{1n}s_{2}+q_{1n}s_{3}+q_{2n}s_{4}+q_{3n}s_{5}& \mathrm{otherwise}
\end{align*}
Since by the lemma \ref{uniqueexp} elements of Ap\'{e}ry set  $\mathrm{Ap}(\Gamma_{5},a)$ have unique expression, we have 
\begin{align*}
o(\mu(n)a+nd)=k &=r_{1n}+q_{1n}+q_{2n}+q_{3n}-1, \quad \mathrm{if}\,\, r_{1n}=2 \,\,\mathrm{and}\,\, q_{3n}>0\\
&=r_{1n}+q_{1n}+q_{2n}+q_{3n}, \quad \mathrm{otherwise}
\end{align*}
Suppose $k\geq 4.$ We claim $q_{3n}=0.$ If not then say  $q_{3n}=0.$ By remark \ref{bound} $k=r_{1n}+q_{1n}+q_{2n}\leq 2$, a contradiction. Therefore we have $q_{3n}\geq 1$. Again by remark \ref{bound}, we get  $q_{3n}\leq k\leq r_{1n}+q_{1n}+q_{2n}+q_{3n}\leq q_{3n}+2 $. So for a given $k,$  $q_{3n}$ have three possibilities viz. $k,$ $k-1$ or $k-2.$   We list down all the elements of order $k.$ \\ \vspace{1cm}
\begin{center}
\begin{tabular}{ |c|c| } 
 \hline
 $q_{3n}$&$ (q_{3n},q_{2n},q_{1n},r_{1n})$\\
\hline
$k$&$(k,0,0,0)$\\
\hline
&$(k-1,1,0,0)$\\
&$(k-1,0,1,0)$\\
$k-1$&$(k-1,0,0,1)$\\
&$(k-1,0,0,2)$\\
\hline
&$(k-2,1,1,0)$\\
&$(k-2,1,0,1)$\\
$k-2$&$(k-2,0,1,1)$\\
&$(k-2,1,0,2)$\\
&$(k-2,0,1,2)$\\
\hline
\end{tabular}

\end{center}
\medskip
Hence we see that there are in total $10$ possible elements of order $k.$ These possibilities are attained provided $n\leq a-1.$ Now $n=10q_{3n}+r_{3n}$ where $r_{3n}\leq 9.$ This shows $q_{3n}\leq q.$ Moreover  $n\leq 10 q_{3n}+9<10(q_{3n}+1).$  We note $k\leq q_{3n}+2\Rightarrow k\leq q+2.$  Hence  if $k<q,$ then $q_{3n}\leq k<q$ showing $n<a.$ Hence all these possibilities are attained in this case.  However not all possiblities are attained for $k=q,q+1,q+2.$ For the case $k=q,$ we see the tuple $(k,0,0,0)$ doesn't occur if $r=0$ since then $q_{3n}<q,$ since $n<a=10q.$ However all the possibilities are retained if $r>0.$ 
For the case $k=q+1$ we note that all possibilities for $q_{3n}=k-2$ occurs since in these cases $n=10q_{3n}+r_{3n}\leq 10(k-2)+9<a.$ Moreover the case $(k,0,0,0)$ never occurs. For the cases with $q_{3n}=k-1$   we see the values of $n$ are from the set
$\{10q+6, 10q+3,10q+1,10q+2\}$ and hence the possibilities arise depending upon the value of $r.$ For example $n=10q+6$ is possible only if $r\geq 7.$ Now the table follows easily for the case $k=q+1.$ 
The case $k=q+2$ is also similar. We must note $\psi(k)=0$ for $k\geq 4.$
 Now for the cases $k=2,3$ the table follows similarly with the only exception that here $q_{3n}$ can be zero for example if we take $n=8$ and in this case $o(\mu(n))=3.$ 
The table for $k=3$ follows by including the extra possiblities $(0,0,1,2),(0,1,0,2).$ For $k=2$ we remove the cases, for which $k=q+2$ since then $q=0$ contrary to our hypothesis. But we need to include the possibilities $(0,1,1,0),(0,1,0,1),(0,0,1,1),(0,0,0,2).$ Hence the number of possibilities turns out to be $10-5+4=10-1=10+\psi(2).$ The table then follows as before.

\end{proof}
\medskip

 We are giving an example of the Ap\'{e}ry table of $\Gamma_{5}$. We use computer algebra software \cite{GAP4} to compute Ap\'{e}ry set and Ap\'{e}ry table of $\Gamma_{5}$.
\begin{example} We give below an example of Ap\'{e}ry table for $a=11,d=2.$ We here note that the Ap\'{e}ry Set is given by $\{0,24,48,39,63,87,56,80,104,95,75\}.$ We first list down the values $q_{3n},q_{2n},q_{1n},r_{1n}$ for $1\leq n\leq 10.$ \begin{center}
\begin{tabular}{ |c|c|c|c|c| c|c|c|c|c|c|} 
 \hline
 $n$&1&2& 3&4&5&6& 7&8&9&10\\  
\hline
 $q_{3n}$&0&0& 0&0&0&0& 0&0&0&1\\ 
\hline 
 $q_{2n}$&0&0& 0&0&0&1& 1&1&1&0\\  
\hline
 $q_{1n}$&0&0& 1&1&1&0& 0&0&1&0\\  
\hline
 $r_{1n}$&1&2& 0&1&2&0& 1&2&0&0\\  
\hline

\end{tabular}
\end{center}

Now we are in a position to write the Ap\'{e}ry table. \begin{center}
\begin{tabular}{ |c|c|c|c|c| c|c|c|c|c|c|c|} 
 \hline
 $Ap(\Gamma_{5})$&0&24& 48&39&63&87& 56&80&104&95&75\\  
\hline

 $Ap(M_{5})$&11&24& 48&39&63&87& 56&80&104&95&75\\  
\hline

 $Ap(2M_{5})$&22&35& 48&50&63&87& 67&91&104&95&86\\  
\hline

 $Ap(3M_{5})$&33&46& 59&61&74&98& 78&91&104&106&97\\  
\hline

\end{tabular}
\end{center}
\end{example}
\medskip
\begin{theorem}\label{aperytable}  Let $\mathrm{AT}(\Gamma_{5},a)$ denote the Ap\'{e}ry table for 
 $\Gamma_{5}$. Then $\mathrm{AT}(\Gamma_{5},a)$ will be of order 
 $(\lfloor\frac{a}{10}\rfloor+2) \times a$.   Let $\omega_{st}$ be the 
 $(s,t)$ entry of the table $\mathrm{AT}(\Gamma_{5},a)$. Then, 
 $$\omega_{st} =\mu(t-1)a+(t-1)d,\quad \mathrm{if}\,\, s\leq o(\mu(t-1)a+(t-1)d)+1,$$ and if $s> o(\mu(t-1)a+(t-1)d)+1,$ then 
 $$\omega_{st}=\mu(t-1)a+(t-1)d +[s-o(\mu(t-1)a+(t-1)d)-1]a $$
 
 Hence the reduction number of $r_{I}(\mathfrak{m})$ is $\lfloor\frac{a}{10}\rfloor+1$.
\end{theorem} 
 
\proof 

By Lemma \ref{number} we see that the maximum order of an element in the Ap\'{e}ry set is $q+2$ or $\lfloor\frac{a}{10}+2\rfloor$ and this maximum is achieved for every case. Hence the order of the table is given by  $(\lfloor\frac{a}{10}\rfloor+2) \times a.$ The rest is just the definition of Ap\'{e}ry table.
\qed 
 \bigskip
 
We take some definitions from \cite{cz2}. Let $W =\{a_{0},\ldots,a_{n}\}$ be a set of integers. We call it a \textit{ladder} if $a_{0}\leq\ldots\leq a_{n}$. Given a ladder, we say that a subset $L=\{a_{i},\ldots,a_{i+k}\}$, with $k\geq 1$, is a \textit{landing} of length $k$ if $a_{i-1}<a_{i}=\cdots=a_{i+k}<a_{i+k+1}$ (where $a_{-1}= -\infty$ and $a_{n+1}=\infty$). In this case, $s(L)=i$ and $e(L)=i+k$. A landing $L$ is said to be a \textit{true landing} if $s(L)\geq 1$. Given two landings $L$ and $L^{'}$, we set $L<L^{'}$ if $s(L)<s(L^{'})$. Let $p(W)+1$ be the number of landings and assume that $L_{0}<\cdots<L_{p(W)}$ are the distinct landings. Then we define the following numbers:
$s_{j}(W)=s(L_{j})$, $e_{j}(W)=e(L_{j})$, for each $0\leq j\leq p(W)$;
$c_{j}(W)=s_{j}(W)-e_{j-1}(W)$, for each $0\leq j\leq p(W)$.
\medskip

Suppose $\Gamma$ be a numerical semigroup minimally 
generated by $a_{1}<\cdots <a_{e}$ and $\mathfrak{m}$ be the maximal ideal of $k[[t^{a_{1}},\ldots t^{a_{e}}]]$. Let $r= r_{(t^{a_{1}})}(\mathfrak{m})$,  $M=\Gamma\setminus\{0\}$ and 
$\mathrm{Ap}(nM)=\{\omega_{n,0},\ldots\omega_{n,a_{1}-1}\}$ for $0\leq n \leq r$. For every $1\leq i\leq a_{1}-1$, consider the ladder of the values $W^{i}=\{\omega_{n,i}\}_{0\leq n\leq r}$ and define the following integers:
\begin{enumerate}[(i)]
\item $p_{i}=p(W^{i})$
\item $d_{i}=e_{p_{i}}(W^{i})$
\item $b_{j}^{i}=e_{j-1}(W^{i})$ and 
$c_{j}^{i}=c_{j}(W^{i})$, for $1\leq j\leq p_{i}$.
\end{enumerate}
\medskip

\begin{theorem}\textbf{(Cortadellas, Zarzuela.)}\label{tangentcone} With the above notations, 
$$G_{\mathfrak{m}}\cong F\oplus\displaystyle\bigoplus_{i=1}^{a_{1}-1}\left(F(-d_{i})\displaystyle \bigoplus_{j=1}^{p_{i}}\dfrac{F}{(({t^{a_{1}})^{*})^{c_{j}^{i}}}F}(-b_{j}^{i})\right),$$
where $G_{\mathfrak{m}}$ is the tangent cone of $\Gamma$ and $F=F((t^{a_{1}}))$ is the fiber cone.
\end{theorem}

\proof See Theorem 2.3 in \cite{cz2}.\qed
\medskip

\begin{corollary}\label{exptangent} The tangent cone $G_{\mathfrak{m}}$ of $\Gamma_{5}$ is a free 
$F(I)$-module. Moreover 
$$ G_{\mathfrak{m}}= \displaystyle\bigoplus_{k=0}^{\lfloor \frac{a}{10}\rfloor+1}(F(I)(-k))^{t_{k}},$$ 
where $t_{k}$'s are given in Lemma \ref{number}. 
\end{corollary}  

\proof Proof follows from corollary \ref{aperytable} and \ref{tangentcone}.\qed
\medskip

\begin{corollary} The following properties hold good 
for the tangent cone $G_{\mathfrak{m}}$ of $\Gamma_{4}$ :
\begin{enumerate}
\item[(i)] $G_{\mathfrak{m}}$ is Cohen-Macaulay;
\item[(ii)] $G_{\mathfrak{m}}$ is not Gorenstein;
\item[(iii)] $G_{\mathfrak{m}}$ is Buchsbaum.
\end{enumerate}
\end{corollary}

\proof Same as in Corollary $6.8$ in \cite{pss1} \qed
\medskip

\begin{corollary} Let $HG_{\mathfrak{m}}(x)$ be the Hilbert series of $G_{\mathfrak{m}}$. Then $$HG_{\mathfrak{m}}(x)=\displaystyle\left(\sum_{k=0}^{\lfloor \frac{a}{10}\rfloor+1} t_{k}x^{k}\right)/(1-x).$$
Where $t_{k}$'s are given in Lemma \ref{number}. 
\end{corollary}
\proof Follows from Corollay \ref{exptangent}.\qed
\medskip

\begin{conjecture}\label{uniquecon}  The Ap\'{e}ry set $\mathrm{Ap}(\Gamma_{m},a)$ have unique expression. Hence the tangent cone of $k[[\Gamma_{m}]]$ is Cohen-Macaulay.
\end{conjecture}
\begin{notation}
 We make the following notations, for $n\in\mathbb{N},$ 
\begin{itemize}
\item $n=15s_{4n}+t_{4n}, $  \  $0\leq t_{4n}\leq 14$
\item $t_{4n}=10s_{3n}+t_{3n}, $  \  $0\leq t_{3n}\leq 9$
\item $t_{3n}= 6s_{2n}+t_{2n}, $  \  $0\leq t_{2n}\leq 5$
\item $t_{2n}=3s_{1n}+t_{1n}, $  \  $0\leq t_{1n}\leq 2$
\item $S=\{12, 20+15k,23+15k,27+15k \mid k\geq 0\}$
\item \begin{align*}
\nu(n)&= 2t_{1n}+3s_{1n}+4s_{2n}+5s_{3n}+6s_{4n},\quad n\notin S\\
&= 2t_{1n}+3s_{1n}+4s_{2n}+5s_{3n}+6s_{4n}-1  \quad n\in S\setminus\{ 20+15k\mid  k\geq 0\} \\
&= 2t_{1n}+3s_{1n}+4s_{2n}+5s_{3n}+6s_{4n}-3,\quad \mathrm{ otherwise.}
\end{align*}
\end{itemize}
\end{notation}
\begin{conjecture}\label{apery6con}
 The Ap\'{e}ry set $\mathrm{Ap}(\Gamma_{6},a)$ will be given by $$\mathrm{Ap}(\Gamma_{6},a)=\{\nu(n)a+nd\mid n\in \{1,\ldots a-1\}\}\cup\{0\}.$$  
\end{conjecture}

 \bibliographystyle{amsalpha}

\end{document}